\documentclass[a4paper,12pt]{article}%
\usepackage{amssymb}
\usepackage{amsmath}
\usepackage{latexsym}
\usepackage{graphicx}
\usepackage[latin1]{inputenc}
\usepackage{amsfonts}
\usepackage{hyperref}
\usepackage{amsthm}
\usepackage{graphicx}
\usepackage{color}
\usepackage{verbatim}%
\providecommand{\U}[1]{\protect\rule{.1in}{.1in}}
\font\teneusb=eusb10 \font\seveneusb=eusb7 \font\fiveeusb=eusb5
\newfam\eusbfam \textfont\eusbfam=\teneusb
\scriptfont\eusbfam=\seveneusb \scriptscriptfont\eusbfam=\fiveeusb

\font\tenbifull=cmmib10
\font\tenbimed=cmmib7
\font\tenbismall=cmmib5
\mathchardef\bbGamma="7000 \mathchardef\bbDelta="7001
\mathchardef\bbPhi="7002 \mathchardef\bbAlpha="7003
\mathchardef\bbXi="7004 \mathchardef\bbPi="7005
\mathchardef\bbSigma="7006 \mathchardef\bbUpsilon="7007
\mathchardef\bbTheta="7008 \mathchardef\bbPsi="7009
\mathchardef\bbOmega="700A \mathchardef\bbalpha="710B
\mathchardef\bbbeta="710C \mathchardef\bbgamma="710D
\mathchardef\bbdelta="710E \mathchardef\bbepsilon="710F
\mathchardef\bbzeta="7110 \mathchardef\bbeta="7111
\mathchardef\bbtheta="7112 \mathchardef\bbiota="7113
\mathchardef\bbkappa="7114 \mathchardef\bblambda="7115
\mathchardef\bbmu="7116 \mathchardef\bbnu="7117
\mathchardef\bbxi="7118 \mathchardef\bbpi="7119
\mathchardef\bbrho="711A \mathchardef\bbsigma="711B
\mathchardef\bbtau="711C \mathchardef\bbupsilon="711D
\mathchardef\bbphi="711E \mathchardef\bbchi="711F
\mathchardef\bbpsi="7120 \mathchardef\bbomega="7121
\mathchardef\bbvarepsilon="7122 \mathchardef\bbvartheta="7123
\mathchardef\bbvarpi="7124 \mathchardef\bbvarrho="7125
\mathchardef\bbvarsigma="7126 \mathchardef\bbvarphi="7127

\newcommand{\N}{{\rm I}\kern-0.18em{\rm N}}

\newcommand{\h}{{\rm I}\kern-0.18em{\rm H}}
\newcommand{\K}{{\rm I}\kern-0.18em{\rm K}}

\newcommand{\Z}{{\rm Z}\kern-0.34em{\rm Z}}

\newcommand{\1}{{\rm 1}\kern-0.22em{\rm I}}

\newtheorem{thm}{Theorem}[section]

\newtheorem{ex}{Example}[section]

\newtheorem{cor}{Corollary}[section]
\newtheorem{theo}{Theorem}[section]
\newtheorem{lem}{Lemma}[section]
\newtheorem{rem}{Remark}[section]

\numberwithin{equation}{section}

\newcounter{eqroman}
\begin{document}

\title{Light tails: Gibbs conditional principle under extreme deviation}
\author{Michel Broniatowski$^{(1)}$ and Zhansheng Cao$^{(1)}$\\LSTA, Université Paris 6}
\maketitle

\begin{abstract}
Let $X_{1},..,X_{n}$ denote an i.i.d. sample with light tail distribution and
$S_{1}^{n}$ denote the sum of its terms; let $a_{n}$ be a real sequence\ going
to infinity with $n.$\ In a previous paper (\cite{BoniaCao}) it is proved that
as $n\rightarrow\infty$, given $\left(  S_{1}^{n}/n>a_{n}\right)  $ all terms
$X_{i_{\text{ }}}$ concentrate around $a_{n}$ with probability going to $1$.
This paper explores the asymptotic distribution of $X_{1}$ under the
conditioning events $\left(  S_{1}^{n}/n=a_{n}\right)  $ and $\left(
S_{1}^{n}/n\geq a_{n}\right)  $ . It is proved that under some regulatity
property, the asymptotic conditional distribution of $X_{1}$ given $\left(
S_{1}^{n}/n=a_{n}\right)  $ can be approximated in variation norm by the
tilted distribution at point $a_{n}$ , extending therefore the classical LDP
case developed in (\cite{Diaconis1}) . Also under $\left(  S_{1}^{n}/n\geq
a_{n}\right)  $ the dominating point property holds.

It also considers the case when the $X_{i}$'s are $\mathbb{R}^{d}-$valued, $f$
is a real valued function defined on $\mathbb{R}^{d}$ and the conditioning
event writes $\left(  U_{1}^{n}/n=a_{n}\right)  $ or $\left(  U_{1}^{n}/n\geq
a_{n}\right)  $ with $U_{1}^{n}:=\left(  f(X_{1})+..+f(X_{n})\right)  /n$ and
$f(X_{1})$ has a light tail distribution$.$ As a by-product some attention is
paid to the estimation of high level sets of functions.

\end{abstract}

\section{\bigskip Introduction}

\label{SectIntro}Let $X_{1},..,X_{n}$ denote $n$ independent unbounded real
valued random variables and $S_{1}^{n}:=X_{1}+..+X_{n}$ be their sum. The
purpose of this paper is to explore the limit distribution of the generic
variable $X_{1}$ conditioned on extreme deviations (ED) pertaining to
$S_{1}^{n}.$ By extreme deviation we mean that $S_{1}^{n}/n$ is supposed to
take values which are going to infinity as $n$ increases. Obviously such
events are of infinitesimal probability. Our interest in this question stems
from a first result which assesses that under appropriate conditions, when the
sequence $a_{n}$ is such that
\[
\lim_{n\rightarrow\infty}a_{n}=\infty
\]
then there exists a sequence $\varepsilon_{n}$ which satisfies $\epsilon
_{n}/a_{n}\rightarrow0$ as $n$ tends to infinity such that
\begin{equation}
\lim_{n\rightarrow\infty}P\left(  \left.  \cap_{i=1}^{n}\left(  X_{i}%
\in\left(  a_{n}-\varepsilon_{n},a_{n}+\varepsilon_{n}\right)  \right)
\right\vert S_{1}^{n}/n\geq a_{n}\right)  =1 \label{democracy}%
\end{equation}
which is to say that when the empirical mean takes exceedingly large values,
then all the summands share the same behaviour; this result is useful when
considering aggregate forming in large random media, or in the context of
robust estimators in statistics. It requires a number of hypotheses, which we
simply quote as of \textquotedblleft light tail" type. We refer to
\cite{BoniaCao} for this result and the connection with earlier related works,
and name it "democratic localization principle" (DLP), as referred to in
\cite{Sornette}.

The above result is clearly to be put in relation with the so-called Gibbs
conditional Principle which we recall briefly in its simplest form.

Let $a_{n}$ satisfy $a_{n}$ $=a$ , a constant with value larger than the
expectation of $X_{1}$ and consider the behaviour of the summands when
$\left(  S_{1}^{n}/n\geq a\right)  $ , under a large deviation (LD) condition
about the empirical mean. The asymptotic conditional distribution of $X_{1}$
given $\left(  S_{1}^{n}/n\geq a\right)  $ is the well known tilted
distribution of $P_{X}$ with parameter $t$ associated to $a.$ Let us introduce
some notation. The hypotheses to be stated now together with notation are kept
throughout the entire paper.

It will be assumed that $P_{X}$ , which is the distribution of $X_{1}$, has a
density $p$ with respect to the Lebesgue measure on $\mathbb{R}$. The fact
that $X_{1}$ has a light tail is captured in the hypothesis that $X_{1}$ has a
moment generating function
\[
\phi(t):=E\exp tX_{1}%
\]
which is finite in a non void neighborhood $\mathcal{N}$ of $0.$ This fact is
usually referred to as a Cramer type condition.

Defined on $\mathcal{N}$ are the following functions. The functions
\[
t\rightarrow m(t):=\frac{d}{dt}\log\phi(t)
\]%
\[
t\rightarrow s^{2}(t):=\frac{d^{2}}{dt^{2}}\log\phi(t)
\]
and%

\[
t\rightarrow\mu_{3}(t):=\frac{d^{3}}{dt^{3}}\log\phi(t)\text{\ }%
\]
are the expectation, the variance and kurtosis of the r.v. $\mathcal{X}_{t}$
with density
\[
\pi^{a}(x):=\frac{\exp tx}{\phi(t)}p(x)
\]
which is defined on $\mathbb{R}$ and which is the tilted density with
parameter $t$ in $\mathcal{N}$ defined through
\begin{equation}
m(t)=a. \label{m(t)=a}%
\end{equation}
$.$ When $\phi$ is steep, meaning that
\[
\lim_{t\rightarrow\partial\mathcal{N}}m(t)=\infty
\]
then $m$ parametrizes the convex hull $cvhull\left(  P_{X}\right)  $ of the
support of $P_{X}$ and (\ref{m(t)=a}) is defined in a unique way for all $a$
in $cvhull\left(  P_{X}\right)  .$ We refer to \cite{Barndorff} for those properties.

We now come to some remark on the Gibbs conditional principle in the standard
above setting. A phrasing of this principle is:

As $n$ tends to infinity the conditional distribution of $X_{1}$ given
$\left(  S_{1}^{n}/n\geq a\right)  $ approaches $\Pi^{a},$ the distribution
with density $\pi^{a}.$

We state the Gibbs principle in a form where the conditioning event is a point
condition $\left(  S_{1}^{n}/n=a\right)  .$ The conditional distribution of
$X_{1}$ given $\left(  S_{1}^{n}/n=a\right)  $ is a well defined distribution
and Gibbs conditional principle states that it converges to $\Pi^{a}$ as $n$
tends to infinity. In both settings, this convergence holds in total variation
norm. We refer to \cite{Diaconis1} for the local form of the conditioning
event; we will mostly be interested in the extension of this form.

The present paper is also a continuation of \cite{BrCaron} which contains a
conditional limit theorem for the approximation of the conditional
distribution of $X_{1},...,X_{k_{n}}$ given $S_{1}^{n}/n=a_{n}$ with $\lim
\sup_{n\rightarrow\infty}k_{n}/n\leq1$ and $\lim_{n\rightarrow\infty}%
n-k_{n}=\infty$. There, the sequence $a_{n}$ is bounded, hence covering all
cases from the LLN up to the LDP, and the approximation holds in the total
variation distance. The resulting approximation, when restricted to the case
$k_{n}=1$, writes%

\begin{equation}
\lim_{n\rightarrow\infty}\int\left\vert p_{a_{n}}(x)-g_{a_{n}}(x)\right\vert
dx)=0 \label{BRCA}%
\end{equation}
where

\bigskip%

\begin{equation}
g_{a_{n}}(x):=Cp(x)\mathfrak{n}\left(  a_{n},s_{n}^{2},x\right)  . \label{Ga}%
\end{equation}
Hereabove $\mathfrak{n}\left(  a,s_{n},x\right)  $ denotes the normal density
function at point $x$ with expectation $a_{n}$, with variance $s_{n}^{2}$, and
$s_{n}^{2}:=s^{2}(t_{n})(n-1)$ and $t_{n}$ such that $m(t_{n})=a_{n};$ $C$ is
a normalizing constant. Obviously developing in display (\ref{Ga}) yields
\[
g_{a_{n}}(x)=\pi^{a_{n}}(x)\left(  1+o(1)\right)
\]
which proves that (\ref{BRCA}) is a form of Gibbs principle, with some
improvement due to the second order term. In the present context the extension
from $k=1$ (or from fixed $k$) to the case when $k_{n}$ approaches $n$
requires a large burden of technicalities, mainly Edgeworth expansions of high
order in the extreme value range; lacking a motivation for this task, we did
not engage on this path.

The paper is organized as follows. Notation and hypotheses are stated in
Section \ref{SctNotationHyp}; a sharp Abelian result pertaining to the moment
generating function and a refinement of a local central limit theorem in the
context of triangular arrays with non standard moments are presented in
Section \ref{SectAbelEdgeworth}. Section \ref{SectGibbs} provides a local
Gibbs conditional principle under EDP, namely producing the approximation of
the conditional density of $X_{1}$ conditionally on $\left(  S_{1}^{n}%
/n=a_{n}\right)  $ for sequences $a_{n}$ which tend to infinity. The first
approximation is local. This result is extended to typical paths under the
conditional sampling scheme, which in turn provides the approximation in
variation norm for the conditional distribution$.$ The method used here
follows closely the approach developed in \cite{BrCaron}. Extensions to other
conditioning events are discussed. The differences between the Gibbs
principles in LDP and EDP are also mentioned. Similar results in the case when
the conditioning event is $\left(  S_{1}^{n}/n\geq a_{n}\right)  $ are stated
in Section \ref{SectExceedances}, where both EDP and DLP are simultaneously
considered. This section also introduces some proposal for a stochastic
approximation of high level sets of real valued functions defined on
$\mathbb{R}^{d}.$

\section{Notation and hypotheses}

\label{SctNotationHyp}The density $p$ of $X_{1}$ is uniformly bounded and
writes
\begin{equation}
p(x)=c\exp\Big(-\big(g(x)-q(x)\big)\Big)\qquad x\in\mathbb{R}_{+},
\label{densityFunction}%
\end{equation}
where $c$ is some positive normalizing constant. Define
\[
h(x):=g^{\prime}(x).
\]
We assume that for some positive constant $\vartheta$ , for large $x$, it
holds
\begin{equation}
\sup_{|v-x|<\vartheta x}|q(v)|\leq\frac{1}{\sqrt{xh(x)}}.
\label{densityFunction01}%
\end{equation}
The function $g$ is positive and satisfies
\begin{equation}
\lim_{n\rightarrow\infty}\frac{g(x)}{x}=\infty. \label{3section101}%
\end{equation}

Not all positive $g$'s satisfying $(\ref{3section101})$ are adapted to our
purpose. Regular functions $g$ are defined through the function $h$ as
follows. We define firstly a subclass $R_{0}$ of the family of \emph{slowly
varying} function. A function $l$ belongs to $R_{0}$ if it can be represented
as
\begin{equation}
l(x)=\exp\Big(\int_{1}^{x}\frac{\epsilon(u)}{u}du\Big),\qquad x\geq1,
\label{3section102}%
\end{equation}
where $\epsilon(x)$ is twice differentiable and $\epsilon(x)\rightarrow0$ as
$x\rightarrow\infty$.

We follow the line developed in $\cite{Nagaev}$ to describe the assumed
regularity conditions of $h$.

\textbf{The} \textbf{Class }${R_{\beta}}$\textbf{ :} $x\rightarrow h(x)$
belongs to ${R_{\beta}}$, if, with $\beta>0$ and $x$ large enough, $h(x)$ can
be represented as
\[
h(x)=x^{\beta}l(x),
\]
where $l(x)\in R_{0}$ and in $(\ref{3section102})$ $\epsilon(x)$ satisfies
\begin{equation}
\limsup_{x\rightarrow\infty}x|\epsilon^{\prime}(x)|<\infty,\qquad
\limsup_{x\rightarrow\infty}x^{2}|\epsilon^{^{\prime\prime}}(x)|<\infty.
\label{3section104}%
\end{equation}

\textbf{The} \textbf{Class }$R$\textbf{${_{\infty}}$ :} $x\rightarrow$ $l(x)$
belongs to $\widetilde{R_{0}}$, if, in $(\ref{3section102})$, $l(x)\rightarrow
\infty$ as $x\rightarrow\infty$ and
\begin{equation}
\lim_{x\rightarrow\infty}\frac{x\epsilon^{\prime}(x)}{\epsilon(x)}%
=0,\qquad\lim_{x\rightarrow\infty}\frac{x^{2}\epsilon^{^{\prime\prime}}%
(x)}{\epsilon(x)}=0, \label{3section103}%
\end{equation}
and, for some $\eta\in(0,1/4)$
\begin{equation}
\liminf_{x\rightarrow\infty}x^{\eta}\epsilon(x)>0. \label{3section1030}%
\end{equation}
We say that $h$ belongs to ${R_{\infty}}$ if $h$ is increasing and strictly
monotone and its inverse function $\psi$ defined through
\begin{equation}
\psi(u):=h^{\leftarrow}(u):=\inf\left\{  x:h(x)\geq u\right\}
\label{inverse de h}%
\end{equation}
belongs to $\widetilde{R_{0}}$.

Denote $\mathfrak{R:}={R_{\beta}}\cup{R_{\infty}}$. The class $\mathfrak{R}$
covers a large collection of functions, although, ${R_{\beta}}$ and
${R_{\infty}}$ are only subsets of\ the classes of \emph{Regularly varying}
and \emph{Rapidly varying} functions, respectively.

\begin{ex}
\textbf{Weibull Density. Let }$p$ be a Weibull density with shape parameter
$k>1$ and scale parameter $1$, namely%
\begin{align*}
p(x)  &  =kx^{k-1}\exp(-x^{k}),\qquad x\geq0\\
&  =k\exp\Big(-\big(x^{k}-(k-1)\log x\big)\Big).
\end{align*}
Take $g(x)=x^{k}-(k-1)\log x$ and $q(x)=0$. Then it holds
\[
h(x)=kx^{k-1}-\frac{k-1}{x}=x^{k-1}\big(k-\frac{k-1}{x^{k}}\big).
\]
Set $l(x)=k-(k-1)/x^{k},x\geq1$, then $(\ref{3section102})$ holds, namely,
\[
l(x)=\exp\Big(\int_{1}^{x}\frac{\epsilon(u)}{u}du\Big),\qquad x\geq1,
\]
with
\[
\epsilon(x)=\frac{k(k-1)}{kx^{k}-(k-1)}.
\]
The function $\epsilon$ is twice differentiable and goes to $0$ as
$x\rightarrow\infty$. Additionally, $\epsilon$ satisfies condition
$(\ref{3section104})$. Hence we have shown that $h\in R_{k-1}$.
\end{ex}

\begin{ex}
\textbf{A rapidly varying density.} Define $p$ through
\[
p(x)=c\exp(-e^{x-1}),\qquad x\geq0.
\]
Then $g(x)=h(x)=e^{x}$ and $q(x)=0$ for all non negative $x$. We show that
$h\in R_{\infty}$. It holds $\psi(x)=\log x+1$. Since $h(x)$ is increasing and
monotone, it remains to show that $\psi(x)\in\widetilde{R_{0}}$. When $x\geq
1$, $\psi(x)$ admits the representation of $(\ref{3section102})$ with
$\epsilon(x)=\log x+1$. Also conditions $(\ref{3section103})$ and
$(\ref{3section1030})$ are satisfied. Thus $h\in R_{\infty}$.
\end{ex}

Throughout the paper we use the following notation. When a r.v. $X$ has
density $p$ we write $p(X=x)$ instead of $\ p(x).$ For example $\pi^{a}(X=x)$
is the density at point $x$ for the variable $X$ generated under $\pi^{a}$,
while $p(X=x)$ states for $X$ generated under $p.$

For all $\alpha$ (depending on $n$ or not) $P_{\alpha}$ designates the
conditional distribution of the vector $X_{1},..,X_{n}$ given $\left(
S_{1}^{n}=n\alpha\right)  .$ This distribution is degenerate on $\mathbb{R}%
^{n};$ however its margins are a.c. w.r.t. the Lebesgue measure on
$\mathbb{R}$. The function $p_{\alpha}$ denotes the density of a margin.\

\section{An Abelian Theorem and an Edgeworth expansion}

\label{SectAbelEdgeworth}In this short section we mention two Theorems to be
used in the derivation of our main result. They deserve interest by
themselves; see \cite{BrCao2} for their proof.

\subsection{An Abelian type result}

We inherit of the definition of the tilted density $\pi^{a}$ defined in
Section \ref{SectIntro}, and of the corresponding definitions of the functions
$m$, $s^{2}$ and $\mu_{3}$. Because of (\ref{densityFunction}) and on the
various conditions on $g$ those functions are defined as $t\rightarrow\infty.$
The proof of Corollary \ref{3cor1} is postponed to the Appendix.

\begin{thm}
\label{order of s} Let $p(x)$ be defined as in $(\ref{densityFunction})$ and
$h(x)\in\mathfrak{R}$. Denote by
\[
m(t)=\frac{d}{dt}\log\phi(t),\quad\quad s^{2}(t)=\frac{d}{dt}m(t),\qquad
\mu_{3}(t)=\frac{d^{3}}{dt^{3}}\log\Phi(t),
\]
then with $\psi$ defined as in (\ref{inverse de h}) it holds as $t\rightarrow
\infty$
\[
m(t)\sim\psi(t),\qquad s^{2}(t)\sim\psi^{\prime}(t),\qquad\mu_{3}(t)\sim
\frac{M_{6}-3}{2}\psi^{^{\prime\prime}}(t),
\]
where $M_{6}$ is the sixth order moment of standard normal distribution.
\end{thm}

\begin{cor}
\label{3cor1} Let $p(x)$ be defined as in $(\ref{densityFunction})$ and
$h(x)\in\mathfrak{R}$. Then it holds as $t\rightarrow\infty$
\begin{align}
\label{3cor10g}\frac{\mu_{3}(t)}{s^{3}(t)}\longrightarrow0.
\end{align}

\end{cor}

For clearness we write $m$ and $s^{2}$ for $m(t)$ and $s^{2}(t)$.

\begin{ex}
\textbf{The Weibull case}: When $g(x)=x^{k}$and $k>1$ then $m(t)\sim
Ct^{1/\left(  k-1\right)  }$ and $s^{2}(t)\sim C^{\prime}t^{\left(
2-k\right)  /\left(  k-1\right)  }$ , which tends to $0$ for $k>2.$
\end{ex}

\begin{ex}
\textbf{A rapidly varying density.} Define $p$ through
\[
p(x)=c\exp(-e^{x-1}),\qquad x\geq0.
\]
Since $\psi(x)=\log x+1$ it follows that $m(t)\sim\log t$ and $s^{2}%
(t)\sim1/t\rightarrow0.$
\end{ex}

\subsection{Edgeworth expansion under extreme normalizing factors}

With $\pi^{a_{n}}$ defined through
\[
\pi^{a_{n}}(x)=\frac{e^{tx}p(x)}{\phi(t)},
\]
and $t$ determined by $a_{n}=m(t)$, define the normalized density of
$\pi^{a_{n}}$ by
\[
\bar{\pi}^{a_{n}}(x)=s\pi^{a_{n}}(sx+a_{n}),
\]
and denote the $n$-convolution of $\bar{\pi}^{a_{n}}(x)$ by $\bar{\pi}%
_{n}^{a_{n}}(x)$. Denote by $\rho_{n}$ the normalized density of
$n$-convolution $\bar{\pi}_{n}^{a_{n}}(x)$,
\[
\rho_{n}(x):=\sqrt{n}\bar{\pi}_{n}^{a_{n}}(\sqrt{n}x).
\]
The following result extends the local Edgeworth expansion of the distribution
of normalized sums of i.i.d. r.v's to the present context, where the summands
are generated under the density $\bar{\pi}^{a_{n}}$. Therefore the setting is
that of a triangular array of row wise independent summands; the fact that
$a_{n}\rightarrow\infty$ makes the situation unusual. The proof of this result
follows Feller's one (Chapiter 16, Theorem 2 $\cite{Feller}$). With
$m(t)=a_{n}$ and $s^{2}:=s^{2}(t)$, the following Edgeworth expansion holds.

\begin{thm}
\label{3theorem1} With the above notation, uniformly upon $x$ it holds
\begin{equation}
\rho_{n}(x)=\phi(x)\Big(1+\frac{\mu_{3}}{6\sqrt{n}s^{3}}\big(x^{3}%
-3x\big)\Big)+o\Big(\frac{1}{\sqrt{n}}\Big). \label{Edgeworth}%
\end{equation}
where $\phi(x)$ is standard normal density.
\end{thm}

\section{Gibbs' conditional principles under extreme events}

\label{SectGibbs}We now explore Gibbs conditional principles under extreme
events. For $Y_{1},..,Y_{n}$ a random vector generated according to the
conditional distribution of the $X_{i}$'s given $\left(  S_{1}^{n}%
=na_{n}\right)  \,\ $\ and with the density $p_{a_{n}}$ defined as its
marginal density we first provide a density $g_{a_{n}}$ on $\mathbb{R}$ such
that
\[
p_{a_{n}}\left(  Y_{1}\right)  =g_{a_{n}}\left(  Y_{1}\right)  \left(
1+R_{n}\right)
\]
where $R_{n}$ is a function of the vector $\left(  Y_{1},..,Y_{n}\right)  $
which goes to $0$ as $n$ tends to infinity. The above statement may also be
written as
\begin{equation}
p_{a_{n}}\left(  y_{1}\right)  =g_{a_{n}}\left(  y_{1}\right)  \left(
1+o_{P_{a_{n}}}(1)\right)  \label{approxSurRandom}%
\end{equation}
where $P_{a_{n}}$ is the joint probability measure of the vector $\left(
Y_{1},..,Y_{n}\right)  $ under the condition $\left(  S_{1}^{n}=na_{n}\right)
;$ note that we designate $p_{a_{n}}$ the marginal density of $P_{a_{n}}$,
which is well defined although $P_{a_{n}}$ is restricted on the plane
$y_{1}+..+y_{n}=a_{n}.$ This statement amounts to provide the approximation on
typical realizations under the conditional sampling scheme. We will deduce
from (\ref{approxSurRandom}) that the $L^{1}$ distance between $p_{a_{n}}$ and
$g_{a_{n}}$ goes to $0$ as $n$ tends to infinity. \ We first derive a local
marginal Gibbs result, from which (\ref{approxSurRandom}) is easily obtained.

\subsection{A local result in $\mathbb{R}$}

Fix $y_{1}$ in $\mathbb{R}$ and define $t$ through
\begin{equation}
m(t):=a_{n}. \label{3lfd01}%
\end{equation}
Define $s^{2}:=s^{2}(t)$.

Consider the following condition%
\begin{equation}
\lim_{t\rightarrow\infty}\frac{\psi(t)^{2}}{\sqrt{n\psi^{\prime}(t)}}=0,
\label{croissance de a}%
\end{equation}

which amounts to state a growth condition on the sequence $a_{n}.$

Define $z_{1}$ through
\begin{equation}
z_{1}=\frac{na_{n}-y_{1}}{s\sqrt{n-1}}. \label{z1}%
\end{equation}

\begin{lem}
\label{3lemma z} Assume that $p(x)$ satisfies $(\ref{densityFunction})$ and
$h(x)\in\mathfrak{R}$. Let $t$ be defined in $(\ref{3lfd01})$. Assume that
$a_{n}\rightarrow\infty$ as $n\rightarrow\infty$ and that
(\ref{croissance de a}) holds. Then
\[
\lim_{n\rightarrow\infty}z_{1}=0.\qquad
\]

\end{lem}

Proof: When $n\rightarrow\infty$, it holds
\[
z_{1}\sim m(t)/(s(t)\sqrt{n}).
\]
From Theorem $\ref{order of s}$, it holds $m(t)\sim{\psi(t)}$ and
$s(t)\sim\sqrt{\psi^{^{\prime}}(t)}$. Hence we have
\[
z_{1}\sim\frac{\psi(t)}{\sqrt{n\psi^{^{\prime}}(t)}}.
\]

Hence
\[
z_{1}^{2}\sim\frac{\psi(t)^{2}}{{n\psi^{^{\prime}}(t)}}=\frac{\psi(t)^{2}%
}{{\sqrt{n}\psi^{^{\prime}}(t)}}\frac{1}{\sqrt{n}}=o\Big(\frac{1}{\sqrt{n}%
}\Big).
\]

\begin{thm}
\label{point conditional density} With the above notation and hypotheses
denoting $m:=m(t)$, assuming (\ref{croissance de a}), it holds
\[
p_{a_{n}}(y_{1})=p(X_{1}=y_{1}|S_{1}^{n}=na_{n})=g_{m}(y_{1}%
)\Big(1+o(1)\Big).
\]
with
\[
g_{m}(y_{1})=\pi^{m}(X_{1}=y_{1}).
\]

\end{thm}

Proof:

But for limit arguments which are specific to the extreme deviation context,
the proof of this local result is classical; see the LDP\ case in
\cite{Diaconis1}.

We make use of the following invariance property:

For all $y_{1}$ and all $\alpha$ in the range of $X_{1}$%
\[
p(X_{1}=y_{1}|S_{1}^{n}=na_{n})=\pi^{\alpha}(X_{1}=y_{1}|S_{1}^{n}=na_{n})
\]
where on the LHS, the r.v's $X_{i}$ 's are sampled i.i.d. under $p$ and on the
RHS, sampled i.i.d. under $\pi^{\alpha}.$

It thus holds%

\begin{align*}
&  p(X_{1}=y_{1}|S_{1}^{n}=na_{n}-y_{1})=\pi^{m}(X_{1}=y_{1}|S_{1}^{n}%
=na_{n})\\
&  =\pi^{m_{1}}(X_{1}=y_{1})\frac{\pi^{m_{1}}(S_{2}^{n}=na_{n}-y_{1})}{\pi
^{m}(S_{1}^{n}=na_{n})}\\
&  =\frac{\sqrt{n}}{\sqrt{n-1}}\pi^{m_{1}}(X_{1}=y_{1})\frac{\widetilde
{\pi_{n-1}}(\frac{m-y_{1}}{s\sqrt{n-1}})}{\widetilde{\pi_{n}}(0)}\\
&  =\frac{\sqrt{n}}{\sqrt{n-1}}\pi^{m_{1}}(X_{1}=y_{1})\frac{\widetilde
{\pi_{n-1}}(z_{1})}{\widetilde{\pi_{n}}(0)},
\end{align*}
where $\widetilde{\pi_{n-1}}$ is the normalized density of $S_{2}^{n}$ under
i.i.d. sampling under $\pi^{m};$correspondingly, $\widetilde{\pi_{n}}$ is the
normalized density of $S_{1}^{n}$ under the same sampling. Note that a r.v.
with density $\pi^{m}$ has expectation $m=a_{n}$ and variance $s^{2}$.

Perform a third-order Edgeworth expansion of $\widetilde{\pi_{n-1}}(z_{1})$,
using Theorem $\ref{3theorem1}$. It follows
\[
\widetilde{\pi_{n-1}}(z_{1})=\phi(z_{1})\Big(1+\frac{\mu_{3}}{6s^{3}\sqrt
{n-1}}(z_{1}^{3}-3z_{1})\Big)+o\Big(\frac{1}{\sqrt{n}}\Big),
\]
The approximation of $\widetilde{\pi_{n}}(0)$ is obtained from
(\ref{Edgeworth}) through
\[
\widetilde{\pi}(0)=\phi(0)\Big(1+o\big(\frac{1}{\sqrt{n}}\big)\Big).
\]
It follows that
\begin{align*}
&  p(X_{1}=y_{1}|S_{1}^{n}=na_{n})\\
&  =\frac{\sqrt{n}}{\sqrt{n-1}}\pi^{m}(X_{1}=y_{i})\frac{\phi(z_{1})}{\phi
(0)}\Big[1+\frac{\mu_{3}}{6s^{3}\sqrt{n-1}}(z_{1}^{3}-3z_{1})+o\Big(\frac
{1}{\sqrt{n}}\Big)\Big]\\
&  =\frac{\sqrt{2\pi n}}{\sqrt{n-1}}\pi^{m}(X_{1}=y_{1}){\phi(z_{1}%
)}\big(1+R_{n}+o(1/\sqrt{n})\big),
\end{align*}
where
\[
R_{n}=\frac{\mu_{3}}{6s^{3}\sqrt{n-1}}(z_{1}^{3}-3z_{1}).
\]

Under condition $(\ref{croissance de a})$, using Lemma $\ref{3lemma z}$, it
holds $z_{1}\rightarrow0$ as $a_{n}\rightarrow\infty$, and under Corollary
$(\ref{3cor1})$, $\mu_{3}/s^{3}\rightarrow0.$ This yields
\[
R_{n}=o\big(1/\sqrt{n}\big),
\]
which gives
\begin{align*}
&  p(X_{1}=y_{1}|S_{1}^{n}=na_{n})=\frac{\sqrt{2\pi n}}{\sqrt{n-1}}\pi
^{m}(X_{1}=y_{1}){\phi(z_{1})}\big(1+o(1/\sqrt{n})\big)\\
&  =\frac{\sqrt{n}}{\sqrt{n-1}}\pi^{m}(X_{1}=y_{1}){\big(1-z_{1}^{2}%
/2+o(z_{1}^{2})\big)}\big(1+o(1/\sqrt{n})\big),
\end{align*}
where we used a Taylor expansion in the second equality. Using once more Lemma
$\ref{3lemma z}$, under conditions $(\ref{croissance de a})$, we have as
$a_{n}\rightarrow\infty$
\[
z_{1}^{2}=o(1/\sqrt{n}),
\]
whence we get
\begin{align*}
p(X_{1}=y_{1}|S_{1}^{n}=na_{n})  &  =\Big(\frac{\sqrt{n}}{\sqrt{n-1}}\pi
^{m}(X_{1}=y_{1})\big(1+o(1/\sqrt{n})\big)\Big)\\
&  =\Big(1+o\big(\frac{1}{\sqrt{n}}\big)\Big)\pi^{m}(X_{1}=y_{1}),
\end{align*}
which completes the proof.

\subsection{Gibbs conditional principle in variation norm}

\subsubsection{\bigskip Strengthening the local approximation}

We now turn to a stronger approximation of $p_{a_{n}}.$ Consider
$Y_{1},..,Y_{n}$ with distribution $P_{a_{n}}$ and $Y_{1}$ with density
$p_{a_{n}}$ and the resulting random variable $p_{a_{n}}\left(  Y_{1}\right)
.$ We prove the following result

\begin{thm}
\label{ThmLocalProba}With all the above notation and hypotheses it holds%
\[
p_{a_{n}}\left(  Y_{1}\right)  =g_{a_{n}}\left(  Y_{1}\right)  \left(
1+R_{n}\right)
\]
where
\[
g_{a_{n}}=\pi^{a_{n}}%
\]
the tilted density at point $a_{n}$ , and where $R_{n}$ is a function of
$Y_{1}$ $,..,Y_{n}$ such that $P_{a_{n}}\left(  \left\vert R_{n}\right\vert
>\delta\sqrt{n}\right)  \rightarrow0$ as $n\rightarrow\infty$ for any positive
$\delta.$
\end{thm}

\bigskip

This result is of greater relevance than the previous one. Indeed under
$P_{a_{n}}$ the r.v. $Y_{1}$ may take large values as $n$ tends to infinity.
At the contrary the approximation of $p_{a_{n}}$ by $g_{a_{n}}$ on any $y_{1}$
in $\mathbb{R}_{+}$ only provides some knowledge on $p_{a_{n}}$ on sets with
smaller and smaller probability under $p_{a_{n}}$ as $n$ increases $.$ Also it
will be proved that as a consequence of the above result, the $L^{1}$ norm
between $p_{a_{n}}$ and $g_{a_{n}}$ goes to $0$ as $n\rightarrow\infty$, a
result out of reach through the aforementioned result.

In order to adapt the proof of Theorem \ref{point conditional density} to the
present setting it is necessary to get some insight on the plausible values of
$Y_{1}$ under $P_{a_{n}}.$ It holds

\begin{lem}
Under $P_{a_{n}}$ it holds%
\[
Y_{1}=O_{P_{a_{n}}}\left(  a_{n}\right)  .
\]

\end{lem}

Proof: Without loss of generality we can assume $Y_{1}>0.$ By Markov Inequality:%

\[
P\left(  \left.  Y_{1}>u\right\vert S_{1}^{n}=na_{n}\right)  \leq
\frac{E\left(  \left.  Y_{1}\right\vert S_{1}^{n}=na_{n}\right)  }{u}%
=\frac{a_{n}}{u}%
\]
\bigskip which goes to $0$ for all $u=u_{n}$ such that lim$_{n\rightarrow
\infty}u_{n}/a_{n}=\infty.$

\bigskip

We now turn back to the proof of Theorem \ref{ThmLocalProba}. Define
\begin{equation}
Z_{1}:=\frac{na_{n}-Y_{1}}{s\sqrt{n-1}} \label{Z1}%
\end{equation}
the natural counterpart of $z_{1}$ as defined in (\ref{z1}).

It holds%
\[
Z_{1}=O_{p_{a_{n}}}\left(  1/\sqrt{n}\right)  .
\]%
\[
P\left(  \left.  X_{1}=Y_{1}\right\vert S_{1}^{n}=na_{n}\right)
=P(X_{1}=Y_{1})\frac{P\left(  S_{2}^{n}=na_{n}-Y_{1}\right)  }{P\left(
S_{1}^{n}=na_{n}\right)  }%
\]
in which the tilting substitution of measures is performed, with tilting
density $\pi^{a_{n}}$, followed by normalization. Now following verbatim the
proof of Theorem \ref{point conditional density} if the growth condition
(\ref{croissance de a}) holds,\ it follows that
\[
P\left(  \left.  X_{1}=Y_{1}\right\vert S_{1}^{n}=na_{n}\right)  =\pi^{a_{n}%
}\left(  Y_{1}\right)  \left(  1+R_{n}\right)
\]
as claimed where the order of magnitude of $R_{n}$ is $o_{P_{a_{n}}}\left(
1/\sqrt{n}\right)  $. We have proved Theorem \ref{ThmLocalProba}.

Denote the conditional probabilities by $P_{a_{n}}$ and $G_{a_{n}}$ on
$\mathbb{R}^{n}$ , which correspond to the marginal density functions
$p_{a_{n}}$ and $g_{a_{n}}$ on $\mathbb{R}$, respectively.

\subsubsection{From approximation in probability to approximation in variation
norm}

We now consider the approximation of the margin of $P_{a_{n}}$ by $G_{a_{n}}$
in variation norm.

The main ingredient is the fact that in the present setting approximation of
$p_{a_{n}}$ by $g_{a_{n}}$ in probability plus some rate implies approximation
of the corresponding measures in variation norm. This approach has been
developed in \cite{BrCaron}; we state a first lemma which states that whether
two densities are equivalent in probability with small relative error when
measured according to the first one, then the same holds under the sampling of
the second.

Let $\mathfrak{R}_{n}$ and $\mathfrak{S}_{n}$ denote two p.m's on
$\mathbb{R}^{n}$ with respective densities $\mathfrak{r}_{n}$ and
$\mathfrak{s}_{n}.$

\begin{lem}
\label{Lemma:commute_from_p_n_to_g_n} Suppose that for some sequence
$\varepsilon_{n}$ which tends to $0$ as $n$ tends to infinity%
\begin{equation}
\mathfrak{r}_{n}\left(  Y_{1}^{n}\right)  =\mathfrak{s}_{n}\left(  Y_{1}%
^{n}\right)  \left(  1+o_{\mathfrak{R}_{n}}(\varepsilon_{n})\right)
\label{p_n equiv g_n under p_n}%
\end{equation}
as $n$ tends to $\infty.$ Then
\begin{equation}
\mathfrak{s}_{n}\left(  Y_{1}^{n}\right)  =\mathfrak{r}_{n}\left(  Y_{1}%
^{n}\right)  \left(  1+o_{\mathfrak{S}_{n}}(\varepsilon_{n})\right)  .
\label{g_n equiv p_n under g_n}%
\end{equation}

\end{lem}

\begin{proof}
Denote
\begin{equation*}
A_{n,\varepsilon_{n}}:=\left\{ y_{1}^{n}:(1-\varepsilon_{n})\mathfrak{s}%
_{n}\left( y_{1}^{n}\right) \leq\mathfrak{r}_{n}\left( y_{1}^{n}\right) \leq%
\mathfrak{s}_{n}\left( y_{1}^{n}\right) (1+\varepsilon_{n})\right\} .
\end{equation*}
It holds for all positive $\delta$%
\begin{equation*}
\lim_{n\rightarrow\infty}\mathfrak{R}_{n}\left(
A_{n,\delta\varepsilon_{n}}\right) =1.
\end{equation*}
Write
\begin{equation*}
\mathfrak{R}_{n}\left( A_{n,\delta\varepsilon_{n}}\right) =\int \mathbf{1}%
_{A_{n,\delta\varepsilon_{n}}}\left( y_{1}^{n}\right) \frac{\mathfrak{r}%
_{n}\left( y_{1}^{n}\right) }{\mathfrak{s}_{n}(y_{1}^{n})}\mathfrak{s}%
_{n}(y_{1}^{n})dy_{1}^{n}.
\end{equation*}
Since
\begin{equation*}
\mathfrak{R}_{n}\left( A_{n,\delta\varepsilon_{n}}\right) \leq
(1+\delta\varepsilon_{n})\mathfrak{S}_{n}\left(
A_{n,\delta\varepsilon_{n}}\right)
\end{equation*}
it follows that
\begin{equation*}
\lim_{n\rightarrow\infty}\mathfrak{S}_{n}\left(
A_{n,\delta\varepsilon_{n}}\right) =1,
\end{equation*}
which proves the claim.
\end{proof}

Applying this Lemma to the present setting yields%
\[
g_{a_{n}}\left(  Y_{1}\right)  =p_{a_{n}}\left(  Y_{1}\right)  \left(
1+o_{G_{a_{n}}}\left(  1/\sqrt{n}\right)  \right)
\]
as $n\rightarrow\infty.$

This fact entails, as in \cite{BrCaron}

\begin{thm}
\label{ThmcvVarTot}Under all the notation and hypotheses above the total
variation norm between the marginal distribution of $P_{a_{n}}$ and $G_{a_{n}%
}$ goes to $0$ as $n\rightarrow\infty.$
\end{thm}

The proof goes as follows

For all $\delta>0$, let
\[
E_{\delta}:=\left\{  y\in\mathbb{R}:\left\vert \frac{p_{a_{n}}\left(
y\right)  -g_{a_{n}}\left(  y\right)  }{g_{a_{n}}\left(  y\right)
}\right\vert <\delta\right\}
\]
which, turning to the marginal distribution of $P_{a_{n}}$ writes
\begin{equation}
\lim_{n\rightarrow\infty}P_{a_{n}}\left(  E_{\delta}\right)  =\lim
_{n\rightarrow\infty}G_{a_{n}}\left(  E_{\delta}\right)  =1.
\label{limPu1,nGu1,n}%
\end{equation}
It holds%
\[
\sup_{C\in\mathcal{B}\left(  \mathbb{R}\right)  }\left\vert P_{a_{n}}\left(
C\cap E_{\delta}\right)  -G_{a_{n}}\left(  C\cap E_{\delta}\right)
\right\vert \leq\delta\sup_{C\in\mathcal{B}\left(  \mathbb{R}\right)  }%
\int_{C\cap E_{\delta}}g_{a_{n}}\left(  y\right)  dy\leq\delta.
\]
By the above result (\ref{limPu1,nGu1,n})
\[
\sup_{C\in\mathcal{B}\left(  \mathbb{R}\right)  }\left\vert P_{a_{n}}\left(
C\cap E_{\delta}\right)  -P_{a_{n}}\left(  C\right)  \right\vert <\eta_{n}%
\]
and
\[
\sup_{C\in\mathcal{B}\left(  \mathbb{R}\right)  }\left\vert G_{a_{n}}\left(
C\cap E_{\delta}\right)  -G_{a_{n}}\left(  C\right)  \right\vert <\eta_{n}%
\]
for some sequence $\eta_{n}\rightarrow0$ ; hence
\[
\sup_{C\in\mathcal{B}\left(  \mathbb{R}\right)  }\left\vert P_{a_{n}}\left(
C\right)  -G_{a_{n}}\left(  C\right)  \right\vert <\delta+2\eta_{n}%
\]
for all positive $\delta,$ which proves the claim.

As a consequence, applying Scheffé's Lemma we have proved%

\[
\int\left\vert p_{a_{n}}-g_{a_{n}}\right\vert dx\rightarrow0\text{ \ as
}n\rightarrow\infty
\]
\bigskip as sought.

\begin{rem}
This result is to be paralleled with Theorem 1.6 in Diaconis and Freedman
\cite{Diaconis1} and Theorem 2.15 in Dembo and Zeitouni \cite{Dembo} which
provide a rate for this convergence in the LDP range.
\end{rem}

\subsection{\bigskip The asymptotic location of $X$ under the conditioned
distribution}

This section intends to provide some insight on the behaviour of $X_{1}$ under
the condition $\left(  S_{1}^{n}=na_{n}\right)  .$ It will be seen that
conditionally on $\left(  S_{1}^{n}=na_{n}\right)  $ the marginal distribution
of the sample concentrates around $a_{n}.$ Let $\mathcal{X}_{t}$ be a r.v.
with density $\pi^{a_{n}}$ where $m(t)=a_{n}$ and $a_{n}$ satisfies
(\ref{croissance de a}). Recall that $E\mathcal{X}_{t}=a_{n}$ $\ $\ and
$Var\mathcal{X}_{t}=s^{2}$. We evaluate the moment generating function of the
normalized variable $\left(  \mathcal{X}_{t}-a_{n}\right)  /s$. It holds%
\[
\log E\exp\lambda\left(  \mathcal{X}_{t}-a_{n}\right)  /s=-\lambda
a_{n}/s+\log\phi\left(  t+\frac{\lambda}{s}\right)  -\log\phi\left(  t\right)
.
\]
A second order Taylor expansion in the above display yields%
\[
\log E\exp\lambda\left(  \mathcal{X}_{t}-a_{n}\right)  /s=\frac{\lambda^{2}%
}{2}\frac{s^{2}\left(  t+\frac{\theta\lambda}{s}\right)  }{s^{2}}%
\]
where $\theta=\theta(t,\lambda)\in\left(  0,1\right)  .$ The following Lemma,
which is proved in the Appendix, states that the function $t\rightarrow s(t)$
is self neglecting. Namely

\begin{lem}
\label{lemmaSselfneglecting}Under the above hypotheses and notation, for any
compact set $K$
\[
\lim_{n\rightarrow\infty}\sup_{u\in K}\frac{s^{2}\left(  t+\frac{u}{s}\right)
}{s^{2}}=1.
\]

\end{lem}

Applying the above Lemma it follows that the normalized r.v's $\left(
\mathcal{X}_{t}-a_{n}\right)  /s$ converge to a standard normal variable
$N(0,1)$ in distribution, as $n\rightarrow\infty.$ This amount to say that
\begin{equation}
\mathcal{X}_{t}=a_{n}+sN(0,1)+o_{\Pi^{a_{n}}}(1).
\label{approxgaussienneasympt}%
\end{equation}
In the above formula (\ref{approxgaussienneasympt}) the remainder term
$o_{\Pi^{a_{n}}}(1)$ is of smaller order than $sN(0,1)$ since the variance of
a r.v. with distribution $\Pi^{a_{n}}$ is $s^{2}$. When $\lim_{t\rightarrow
\infty}s(t)=0$ then $\mathcal{X}_{t}$ concentrates around $a_{n}$ with rate
$s(t).$ Due to Theorem \ref{ThmcvVarTot} the same holds for $X_{1}$ under
$\left(  S_{1}^{n}=na_{n}\right)  .$This is indeed the case when $g$ is a
regularly varying function with index $\gamma$ larger than $2$, plus some
extra regularity conditions captured in the fact that $h$ belongs to
$R_{\gamma-1}.$ The gausssian tail corresponds to $\gamma=2$ and the variance
of the tilted\ distribution $\Pi^{a_{n}}$ has a non degenerate variance for
\ all $a_{n}$ (in the standard gaussian case, $\Pi^{a_{n}}=N\left(
a_{n},1\right)  $), as does the conditional distribution of $X_{1}$ given
$\left(  S_{1}^{n}=na_{n}\right)  .$

\bigskip

\subsection{Extension to other conditioning events}

Let $X_{1},..,X_{n}$ be $n$ i.i.d. r.v's with common density $p$ defined on
$\mathbb{R}^{d}$ and let $f$ denote a measurable function from $\mathbb{R}%
^{d}$ onto $\mathbb{R}$ such that $f(X_{1})$ has a density $p_{f}$ .We assume
that $p_{f}$ \ enjoys all properties stated in Section \ref{SctNotationHyp} ,
i.e. $p_{f}(x)=\exp-\left(  g(x)-q(x)\right)  $ and we denote accordingly
$\phi_{f}(t)$ its moment generating function, and $m_{f}(t)$ and $s_{f}%
^{2}(t)$ the corresponding first and second derivatives of $\log\phi_{f}(t).$
The following extension of Theorem \ref{ThmLocalProba} holds. Denote
$\Sigma_{1}^{n}:=f(X_{1})+..+f(X_{n}).$

Denote for all $a$ in the range of $f$
\begin{equation}
\pi_{f}^{a}(x):=\frac{\exp tf(x)}{\phi_{f}(t)}p(x) \label{Pif}%
\end{equation}
where $t$ is the unique solution of $m(t):=\left(  d/dt\right)  \log\phi
_{f}(t)=a$ and $\Pi_{f}^{a}$ the corresponding probability measure$.$ Denote
$P_{f,a_{n}}$ the conditional distribution of $\left(  X_{1},..,X_{n}\right)
$ given $\left(  \Sigma_{1}^{n}=na_{n}\right)  .$ The sequence $a_{n}$ is
assume to satisfy (\ref{croissance de a}) where $\psi$ is defined with respect
to $p_{f}.$

\begin{thm}
\label{thmSousContrLin} Under the current hypotheses of this section\ the
variation distance between the margin of $P_{f,a_{n}}$ and $\Pi_{f}^{a_{n}}$
goes to $0$ as $n$ tends to infinity.
\end{thm}

Proof: It holds , for $Y_{1}$ a r.v. with density $p,$%
\begin{align*}
p\left(  \left.  X_{1}=Y_{1}\right\vert \Sigma_{1}^{n}=na_{n}\right)   &
=p\left(  X_{1}=Y_{1}\right)  \frac{p\left(  \Sigma_{2}^{n}=na_{n}%
-f(Y_{1})\right)  }{p\left(  \Sigma_{1}^{n}=na_{n}\right)  }\\
&  =\frac{p\left(  X_{1}=Y_{1}\right)  }{p\left(  f(X_{1})=f(Y_{1})\right)
}p\left(  f(X_{1})=f(Y_{1})\right)  \frac{p\left(  \Sigma_{2}^{n}%
=na_{n}-f(Y_{1})\right)  }{p\left(  \Sigma_{1}^{n}=na_{n}\right)  }\\
&  =\frac{p\left(  X_{1}=Y_{1}\right)  }{p\left(  f(X_{1})=f(Y_{1})\right)
}\pi_{f}^{m}\left(  f(X_{1})=f(Y_{1})\right)  \frac{\pi_{f}^{m}\left(
\Sigma_{2}^{n}=na_{n}-f(Y_{1})\right)  }{\pi_{f}^{m}\left(  \Sigma_{1}%
^{n}=na_{n}\right)  }\\
&  =p\left(  X_{1}=Y_{1}\right)  \frac{e^{tf(Y_{1})}}{\phi_{f}(t)}\frac
{\pi_{f}^{m}\left(  \Sigma_{2}^{n}=na_{n}-f(Y_{1})\right)  }{\pi_{f}%
^{m}\left(  \Sigma_{1}^{n}=na_{n}\right)  }%
\end{align*}
where $m:=\left(  d/dt\right)  \log\phi_{f}(t)=a_{n}.$The proof then follows
verbatim that of Theorem \ref{ThmLocalProba} with $X_{i}$ substituted by
$f(X_{i}).$\bigskip

\subsection{Differences between Gibbs principle under LDP and under extreme
deviation}

It is of interest to confront the present results with the general form of the
Gibbs principle under linear constraints in the LDP range.

Consider the application of the above result to r.v's $Y_{1},..,Y_{n}$ with
$Y_{i}:=\left(  X_{i}\right)  ^{2}$ \ where the $X_{i}$'s are i.i.d. and are
such that the density of the i.i.d. r.v's $Y_{i}$'s satisfy
(\ref{densityFunction}) with all the hypotheses stated in this paper, assuming
that $Y$ has a Weibull distribution with parameter larger than 2 . By the
Gibbs conditional principle under a point conditioning (see e.g.
\cite{Diaconis1}), for \textit{fixed} $a$, conditionally on $\left(
\sum_{i=1}^{n}Y_{i}=na\right)  $ the generic r.v. $Y_{1}$ has a non degenerate
limit distribution
\[
p_{Y}^{\ast}(y):=\frac{\exp ty}{E\exp tY_{1}}p_{Y}(y)
\]
and the limit density of $X_{1}$ under $\left(  \sum_{i=1}^{n}X_{i}%
^{2}=na\right)  $ is%

\[
p_{X}^{\ast}(x):=\frac{\exp tx^{2}}{E\exp tX_{1}^{2}}p_{X}(x)
\]
a non degenerate distribution, with $m_{Y}(t)=a$ . As a consequence of the
above result, when $a_{n}\rightarrow\infty$ the distribution of $X_{1}$ under
the condition $\left(  \sum_{i=1}^{n}X_{i}^{2}=na_{n}\right)  $ concentrates
sharply at $-\sqrt{a_{n}}$ and $+\sqrt{a_{n}}.$

\section{EDP under exceedances}

\label{SectExceedances}

\subsection{DLP and EDP}

\label{EDPandLDP}This section extends the previous Theorem \ref{ThmLocalProba}
when the conditioning event has non null measure and writes $A_{n}=\left(
S_{1}^{n}\geq na_{n}\right)  .$ It also provides a bridge between the
"democratic localization principle (DLP)" (\ref{democracy}) and the
conditional Gibbs principle. We denote the density of $X_{1}$ given $A_{n}$ by
$p_{A_{n}}$ to differentiate it from $p_{a_{n}}$.

In the LDP case, when $a_{n}=a>EX_{1}$ then the distribution of $X_{1}$ given
$\left(  S_{1}^{n}\geq na\right)  $ or given $\left(  S_{1}^{n}=na\right)  $
coincide asymptotically, both converging to the tilted distribution at point
$a$, the dominating point of $\left[  a,\infty\right)  .$ This result follows
from the local approach derived in \cite{Diaconis1}, which uses limit results
for sums of i.i.d. r.v's, and from Sanov Theorem which amounts to identify the
Kullback-Leibler projection of the p.m. $P$ on the set of all p.m's with
expectation larger or equal $a.$ When $a$ is allowed to tend to infinity with
$n$ no Sanov-type result is available presently, and the first approach,
although cumbersome, has to be used.\ The "dominating point property" still holds.

Most applications of extreme deviation principles deal with phenomenons driven
by multiplicative cascade processes with power laws; see e.g. \cite{Sornette}
for fragmentation models and turbulence. For sake of simplicity we restrict to
this case, assuming therefore that $h$ belongs to $R_{k-1}$ for some $k>1.$
All conditions and notation of Section \ref{SctNotationHyp} are assumed to
hold. Furthermore, since we consider the asymptotics of the marginal
distribution of the sample under $A_{n}$, it is natural to assume that $a_{n}$
is such that the democratic localization principle holds together with the
local Gibbs conditional principle.

Specialized to the present setting it holds (see Theorem 6 in \cite{BoniaCao})

\begin{thm}
\label{thmLocalization} Assume that for some $\delta>0$
\begin{equation}
\lim\inf_{n\rightarrow\infty}\frac{\log g(a_{n})}{\log n}>\delta
\label{condition31}%
\end{equation}
and define a sequence $\epsilon_{n}$ such that
\begin{equation}
\lim_{n\rightarrow\infty}\frac{n\log a_{n}}{a_{n}^{k-2}\epsilon_{n}^{2}}=0.
\label{example3}%
\end{equation}
Then
\begin{equation}
\lim_{n\rightarrow\infty}P\left(  \left.
{\displaystyle\bigcap\limits_{i=1}^{n}}
\left(  a_{n}-\epsilon_{n}<X_{i}<a_{n}+\epsilon_{n}\right)  \right\vert
S_{1}^{n}/n\geq a_{n}\right)  =1. \label{PRINCIPEDEMO}%
\end{equation}
Furthermore, when $1<k\leq2$ then $\epsilon_{n}$ can be chosen to satisfy
$\epsilon_{n}/a_{n}\rightarrow0$, and when $k>2$, $\epsilon_{n}$ can be chosen
to satisfy $\epsilon_{n}\rightarrow0.$
\end{thm}

We assume that both conditions (\ref{condition31}) and (\ref{croissance de a})
hold. Therefore we assume that the range of $a_{n}$ makes both the DLP and EDP
hold. This is the case for example when%
\[
a_{n}=n^{\alpha}%
\]
together with
\[
\alpha<\frac{2}{2+k}.
\]

We state the following extension of Theorems \ref{thmLocalization} and
\ref{ThmcvVarTot}.

\begin{thm}
\label{thmExceedances}Assume (\ref{condition31}) and (\ref{example3}). Then
for any family of Borel sets $B_{n}$ such that
\[
\lim\inf_{n\rightarrow\infty}P_{A_{n}}\left(  B_{n}\right)  >0
\]
it holds%
\[
P_{A_{n}}\left(  B_{n}\right)  =(1+o(1))G_{a_{n}}\left(  B_{n}\right)
\]
as $n\rightarrow\infty.$
\end{thm}

\begin{rem}
\bigskip This Theorem is of the same kind as those related to conditioning on
thin sets,\ in the range of the LDP; see \cite{CattGozl}.\
\end{rem}

Proof of Theorem \ref{thmExceedances}:

For the purpose of the proof, we need the following lemma, based on Theorem
$6.2.1$ of Jensen \cite{Jensen}, in order to provide the asymptotic estimation
of the tail probability $P(S_{1}^{n}\geq na_{n})$. Its proof is differed to
the Appendix.

Define
\begin{align}
\label{3virg01000}I(x):=xm^{-1}(x)-\log\Phi\big(m^{-1}(x)\big).
\end{align}

\bigskip and let $t_{n}$ be defined through
\[
m(t_{n})=a_{n}.
\]

\begin{lem}
\label{JensenLemme} Let $X_{1},...,X_{n}$ be i.i.d. random variables with
density $p$ defined in $(\ref{densityFunction})$ and $h\in\mathfrak{R}$.
Suppose that when $n\rightarrow\infty$, it holds
\[
\frac{\psi(t_{n})^{2}}{\sqrt{n}\psi^{\prime}(t_{n})}\longrightarrow0.
\]
Then it holds
\begin{equation}
P(S_{1}^{n}\geq na_{n})=\frac{\exp(-nI(a_{n}))}{\sqrt{2\pi}\sqrt{n}%
t_{n}s(t_{n})}\Big(1+o\big(\frac{1}{\sqrt{n}}\big)\Big). \label{jensen}%
\end{equation}

\end{lem}

Denote $\Delta_{n}:=\left(  a_{n}-\epsilon_{n},a_{n}+\epsilon_{n}\right)  $
and $\Delta_{n}^{+}:=\left(  a_{n},a_{n}+\epsilon_{n}\right)  $. It holds%
\[
P\left(  \left.  S_{1}^{n}/n\in\Delta_{n}\right\vert S_{1}^{n}/n\geq
a_{n}\right)  \geq P\left(  \left.
{\displaystyle\bigcap\limits_{i=1}^{n}}
\left(  X_{i}\in\Delta_{n}\right)  \right\vert S_{1}^{n}/n\geq a_{n}\right)
\]
hence by Theorem \ref{thmLocalization}%
\[
P\left(  \left.  S_{1}^{n}/n\in\Delta_{n}^{+}\right\vert S_{1}^{n}/n\geq
a_{n}\right)  =P\left(  \left.  S_{1}^{n}/n\in\Delta_{n}\right\vert S_{1}%
^{n}/n\geq a_{n}\right)  \rightarrow1
\]
which yields%
\[
\frac{P\left(  S_{1}^{n}/n\in\Delta_{n}^{+}\right)  }{P\left(  S_{1}%
^{n}/n>a_{n}\right)  }\rightarrow1.
\]
Since $\left(  S_{1}^{n}/n\in\Delta_{n}^{+}\right)  \subset$ $\left(
S_{1}^{n}/n>a_{n}\right)  $ it follows that for any Borel set $B$ (depending
on $n$ or not)
\[
P\left(  \left.  X_{1}\in B\right\vert S_{1}^{n}/n>a_{n}\right)
=(1+o(1))P\left(  \left.  X_{1}\in B\right\vert S_{1}^{n}/n\in\Delta_{n}%
^{+}\right)  +C_{n}%
\]
where%
\[
0<C_{n}<\frac{P\left(  S_{1}^{n}/n\geq a_{n}+\epsilon_{n}\right)  }{P\left(
S_{1}^{n}/n\geq a_{n}\right)  }.
\]
Using Lemma \ref{JensenLemme} both in the numerator and the denominator, with
some control using (\ref{condition31}) and (\ref{example3}) proving that
(\ref{jensen}) holds with $a_{n}$ substituted by $a_{n}+\epsilon_{n}$, we get
$C_{n}\rightarrow0$.

\bigskip By Theorem \ref{thmLocalization}, with $B=\Delta_{n}$%
\[
P\left(  \left.  X_{1}\in\Delta_{n}\right\vert S_{1}^{n}/n\geq a_{n}\right)
\rightarrow1
\]
which in turn implies that
\[
P\left(  \left.  X_{1}\in\Delta_{n}\right\vert S_{1}^{n}/n\in\Delta_{n}%
^{+}\right)  \rightarrow1
\]
which proves that the conditional distribution of $X_{1}$ given $S_{1}%
^{n}/n\geq a_{n}$ concentrates on $\Delta_{n}$ , as does its distribution
given $S_{1}^{n}/n\in\Delta_{n}^{+}.$ Furthermore both distributions are
asymptotically equivalent in the sense that
\[
\frac{P\left(  \left.  X_{1}\in B_{n}\right\vert S_{1}^{n}/n\in\Delta_{n}%
^{+}\right)  }{P\left(  \left.  X_{1}\in B_{n}\right\vert S_{1}^{n}/n\geq
a_{n}\right)  }\rightarrow1
\]
for all sequence of Borel sets $B_{n}$ such that
\[
\lim\inf_{n\rightarrow\infty}P\left(  \left.  X_{1}\in B_{n}\right\vert
S_{1}^{n}/n\geq a_{n}\right)  >0.
\]
$.$

\bigskip We now prove Theorem \ref{thmExceedances}.

For a Borel set $B=B_{n}$ it holds%
\begin{align*}
P_{A_{n}}\left(  B\right)   &  =\int_{a_{n}}^{\infty}P_{v}(B)p\left(  \left.
S_{1}^{n}/n=v\right\vert S_{1}^{n}/n>a_{n}\right)  dv\\
&  =\frac{1}{P\left(  S_{1}^{n}/n>a_{n}\right)  }\int_{a_{n}}^{\infty}%
P_{v}(B)p\left(  S_{1}^{n}/n=v\right)  dv\\
&  =(1+o(1))\frac{1}{P\left(  S_{1}^{n}/n\in\Delta_{n}^{+}\right)  }%
\int_{a_{n}}^{\infty}P_{v}(B)p\left(  S_{1}^{n}/n=v\right)  dv\\
&  =(1+o(1))\frac{1}{P\left(  S_{1}^{n}/n\in\Delta_{n}^{+}\right)  }%
\int_{a_{n}}^{a_{n}+\epsilon_{n}}P_{v}(B)p\left(  S_{1}^{n}/n=v\right)  dv\\
&  +(1+o(1))\frac{1}{P\left(  S_{1}^{n}/n\in\Delta_{n}^{+}\right)  }%
\int_{a_{n}+\epsilon_{n}}^{\infty}P_{v}(B)p\left(  S_{1}^{n}/n=v\right)  dv\\
&  =(1+o(1))\int_{a_{n}}^{a_{n}+\epsilon_{n}}P_{v}(B)p\left(  \left.
S_{1}^{n}/n=v\right\vert S_{1}^{n}/n\in\Delta_{n}^{+}\right)  dv\\
&  +(1+o(1))\frac{1}{P\left(  S_{1}^{n}/n\in\Delta_{n}^{+}\right)  }%
\int_{a_{n}+\epsilon_{n}}^{\infty}P_{v}(B)p\left(  S_{1}^{n}/n=v\right)  dv\\
&  =(1+o(1))P_{a_{n}+\theta_{n}\epsilon_{n}}(B)\\
&  +(1+o(1))\frac{1}{P\left(  S_{1}^{n}/n\in\Delta_{n}^{+}\right)  }%
\int_{a_{n}+\epsilon_{n}}^{\infty}P_{v}(B)p\left(  S_{1}^{n}/n=v\right)  dv\\
&  =(1+o(1))P_{a_{n}+\theta_{n}\epsilon_{n}}(B)+C_{n}%
\end{align*}
for some $\theta_{n}$ in $\left(  0,1\right)  .$ The term $C_{n}$ is less than
$P(S_{1}^{n}/n>a_{n}+\epsilon_{n})/P(S_{1}^{n}/n>a_{n})$ which tends to $0$ as
$n$ tends to infinity, under\ (\ref{condition31}) and (\ref{example3}), using
Lemma \ref{JensenLemme}.

Hence when $B_{n}$ is such that $\lim\inf_{n\rightarrow\infty}P_{a_{n}%
+\theta_{n}\epsilon_{n}}(B_{n})>0$ using Theorem \ref{ThmcvVarTot}\ it holds,
since $a_{n}^{\prime}:=a_{n}+\theta_{n}\epsilon_{n}$ satisfies
(\ref{croissance de a})
\begin{equation}
P_{A_{n}}(B_{n})=\left(  1+o(1)\right)  G_{a_{n}+\theta_{n}\epsilon_{n}}%
(B_{n}). \label{DistP_A_nGA_n}%
\end{equation}
It remains to prove that
\begin{equation}
G_{a_{n}+\theta_{n}\epsilon_{n}}(B_{n})=\left(  1+o(1)\right)  G_{a_{n}}%
(B_{n}). \label{sol}%
\end{equation}

Make use of (\ref{approxgaussienneasympt}) to prove the claim. Define $t$
through $m(t)=a_{n}$ and $t^{\prime}$ through $m(t^{\prime})=a_{n}^{\prime}$.
It is enough to obtain%
\begin{equation}
s(t)=\left(  1+o(1)\right)  s(t^{\prime}) \label{s(t)eqs(t')}%
\end{equation}
as $n\rightarrow\infty.$ In the present case when $h$ belongs to $R_{k-1}$,
making use of Theorem \ref{order of s} it holds%
\[
s^{2}(t)=C\left(  1+o(1)\right)  m(t)^{2-k}%
\]
for some constant $C$; a similar formula holds for $s^{2}(t^{\prime})$ which
together yield (\ref{s(t)eqs(t')}), whatever $k>1$.

\bigskip This proves (\ref{sol}) and concludes the proof of Theorem
\ref{thmExceedances}.

\begin{rem}
By (\ref{approxgaussienneasympt}) and Theorem \ref{thmExceedances} the r.v.
$X_{1}$ conditioned upon $\left(  S_{1}^{n}/n\geq a_{n}\right)  $ has a
standard deviation of order $Ca_{n}^{2-k},$ much smaller than $\epsilon_{n}$
which by (\ref{example3}) is larger than $\sqrt{n\log a_{n}}/a_{n}^{\left(
k-2\right)  /2};$this proves that (\ref{PRINCIPEDEMO}) might be improved.
However the qualitative bound $k=2$ appears both in the DLP and in the EDP.
\end{rem}

\subsection{High level sets of functions}

We explore the consequences of Theorem \ref{thmExceedances} in connection with
(\ref{democracy}) in the context when the conditioning event writes
\[
A_{n}:=\frac{1}{n}\sum_{i=1}^{n}f(X_{i})\geq a_{n}%
\]
where the i.i.d. r.v's $X_{1},..,X_{n}$ belong to $\mathbb{R}^{d}$, and $f$ is
a real valued function such that the density $p_{f}$ of $f(X_{1})$ satisfies
all the hypotheses of Section \ref{SctNotationHyp}. We denote $P_{A_{n}}$ the
distribution of $\left(  X_{1},..,X_{n}\right)  $ given $A_{n}$. We assume
that $p^{f}$ satisfies (\ref{densityFunction}) with $h$ in $R_{k-1},$ $k>1.$

Using the DLP stated in Section \ref{EDPandLDP} it holds%

\[
P_{A_{n}}\left(
{\displaystyle\bigcap\limits_{i=1}^{n}}
\left(  f\left(  X_{i}\right)  \in\left(  a_{n}-\epsilon_{n},a_{n}%
+\epsilon_{n}\right)  \right)  \right)  \rightarrow1
\]
which implies that
\[
P_{A_{n}}\left(  f\left(  X_{1}\right)  \in\left(  a_{n}-\epsilon_{n}%
,a_{n}+\epsilon_{n}\right)  \right)  \rightarrow1.
\]
In turn, using Theorem \ref{thmExceedances}
\[
G_{a_{n}}\left(  f\left(  X_{1}\right)  \in\left(  a_{n}-\epsilon_{n}%
,a_{n}+\epsilon_{n}\right)  \right)  \rightarrow1
\]
where $G_{a_{n}}$ has density $g_{a_{n}}$ defined through%
\[
g_{a_{n}}(y):=\frac{e^{ty}}{\phi^{f}(t)}p^{f}(y)
\]
with $\phi^{f}(t):=Ee^{tf(X)}$and $m^{f}(t)=a_{n}.$ Note that following
(\ref{approxgaussienneasympt}), $G_{a_{n}}$ is nearly gaussian with
expectation $a_{n}$ and standard deviation $s^{f}(t):=\left(  d/dt\right)
m^{f}(t).$

Turning back to Theorem \ref{thmSousContrLin} it then holds

\begin{theo}
When the above hypotheses hold let $a_{n}$\ satisfies
\[
\lim\inf_{n\rightarrow\infty}\frac{\log g(a_{n})}{\log n}>\delta
\]
with $h$ in $R_{k-1},$ $k>1$ and
\[
g_{a_{n}}(x):=\frac{e^{tf(x)}}{\phi^{f}(t)}p(x)
\]
where $m^{f}(t)=a_{n}$ . Then there exists $\ $\ a sequence $\epsilon_{n}$
satisfying $\lim_{n\rightarrow\infty}\epsilon_{n}/n=0$ such that a r.v. $X$ in
$\mathbb{R}^{d}$ with density $g_{a_{n}}$ is a solution of the inequation
\[
a_{n}-\epsilon_{n}<f(x)<a_{n}+\epsilon_{n}%
\]
with probability $1$ as $n\rightarrow\infty.$ The sequence $\epsilon_{n}$
satisfies
\[
\lim_{n\rightarrow\infty}\frac{n\log a_{n}}{a_{n}^{k-2}\epsilon_{n}^{2}}=0.
\]

\end{theo}

\section{\bigskip Appendix}

\label{SectAppendix}

\subsection{Proof of Lemma \ref{lemmaSselfneglecting}}

\textbf{Case 1:} if $h\in R_{\beta}$. By Theorem $\ref{order of s}$, it holds
$s^{2}\sim\psi^{\prime}(t)$ with $\psi(t)\sim t^{1/\beta}l_{1}(t)$, where
$l_{1}$ is some slowly varying function. It also holds $\psi^{\prime
}(t)=1/h^{^{\prime}}\big(\psi(t)\big);$ therefore
\begin{align*}
\frac{1}{s^{2}}  &  \sim h^{^{\prime}}\big(\psi(t)\big)=\psi(t)^{\beta-1}%
l_{0}\big(\psi(t)\big)\big(\beta+\epsilon\big(\psi(t)\big)\big)\\
&  \sim\beta t^{1-1/\beta}l_{1}(t)^{\beta-1}l_{0}\big(\psi(t)\big)=o(t),
\end{align*}
which implies that for any $u\in K$
\[
\frac{u}{s}=o(t).
\]
It follows that%

\begin{align*}
\frac{s^{2}\left(  t+u/s\right)  }{s^{2}}  &  \sim\frac{\psi^{\prime}%
(t+u/s)}{\psi^{\prime}(t)}=\frac{\psi(t)^{\beta-1}l_{0}\big(\psi
(t)\big)\big(\beta+\epsilon\big(\psi(t)\big)\big)}{\big(\psi
(t+u/s)\big)^{\beta-1}l_{0}\big(\psi(t+u/s)\big)\big(\beta+\epsilon
\big(\psi(t+u/s)\big)\big)}\\
&  \sim\frac{\psi(t)^{\beta-1}}{\psi(t+u/s)^{\beta-1}}\sim\frac{t^{1-1/\beta
}l_{1}(t)^{\beta-1}}{(t+u/s)^{1-1/\beta}l_{1}(t+u/s)^{\beta-1}}\longrightarrow
1.
\end{align*}

\textbf{Case 2:} if $h\in R_{\infty}$. Then we have $\psi(t)\in\widetilde
{R_{0}};$ hence it holds
\[
\frac{1}{st}\sim\frac{1}{t\sqrt{\psi^{\prime}(t)}}=\sqrt{\frac{1}%
{t\psi(t)\epsilon(t)}}\longrightarrow0.
\]
Hence for any $u\in K$, we get as $n\rightarrow\infty$
\[
\frac{u}{s}=o(t),
\]
thus using the slowly varying propriety of $\psi(t)$ we have
\begin{align*}
\frac{s^{2}\left(  t+u/s\right)  }{s^{2}}  &  \sim\frac{\psi^{\prime}%
(t+u/s)}{\psi^{\prime}(t)}=\frac{\psi(t+u/s)\epsilon(t+u/s)}{t+u/s}\frac
{t}{\psi(t)\epsilon(t)}\\
&  \sim\frac{\epsilon(t+u/s)}{\epsilon(t)}=\frac{\epsilon(t)+O\big(\epsilon
^{\prime}(t)u/s\big)}{\epsilon(t)}\longrightarrow1,
\end{align*}
where we used a Taylor expansion in the second line. This completes the proof.

\bigskip

\subsection{\bigskip Proof of Lemma \ref{JensenLemme}}

For the density $p$ defined in $(\ref{densityFunction})$, we show that $g$ is
convex when $x$ is large. If $h\in R_{\beta}$, it holds for $x$ large
\[
g^{^{\prime\prime}}(x)=h^{^{\prime}}(x)=\frac{h(x)}{x}\big(\beta
+\epsilon(x)\big)>0.
\]
If $h\in R_{\infty}$, its reciprocal function $\psi(x)$ belongs to
$\widetilde{R_{0}}$. Set $x:=\psi(u)$; for $x$ large
\[
g^{^{\prime\prime}}(x)=h^{^{\prime}}(x)=\frac{1}{\psi^{^{\prime}}(u)}=\frac
{u}{\psi(u)\epsilon(u)}>0,
\]
where the inequality holds since $\epsilon(u)>0$ under condition
$(\ref{3section1030}),$ for large $u$. Therefore $g$ is convex for large $x$ .

Therefore, the density $p$ with $h\in\mathfrak{R}$ satisfies the conditions of
Jensen's Theorem 6.2.1 (\cite{Jensen}). A third order Edgeworth expansion in
formula $(2.2.6)$ of $(\cite{Jensen})$ yields
\begin{equation}
P(S_{1}^{n}\geq na_{n})=\frac{\Phi(t_{n})^{n}\exp(-nt_{n}a_{n})}{\sqrt{n}%
t_{n}s(t_{n})}\Big(B_{0}(\lambda_{n})+O\big(\frac{\mu_{3}(t_{n})}{6\sqrt
{n}s^{3}(t_{n})}B_{3}(\lambda_{n})\big)\Big), \label{E3}%
\end{equation}
where $\lambda_{n}=\sqrt{n}t_{n}s(t_{n})$, and $B_{0}(\lambda_{n})$ and
$B_{3}(\lambda_{n})$ are defined by
\[
B_{0}(\lambda_{n})=\frac{1}{\sqrt{2\pi}}\Big(1-\frac{1}{\lambda_{n}^{2}%
}+o(\frac{1}{\lambda_{n}^{2}})\Big),\qquad B_{3}(\lambda_{n})\sim-\frac
{3}{\sqrt{2\pi}\lambda_{n}}.
\]
We show that, as $a_{n}\rightarrow\infty,$
\begin{equation}
\frac{1}{\lambda_{n}^{2}}=o\left(  \frac{1}{n}\right)  . \label{E1}%
\end{equation}
Since $n/\lambda_{n}^{2}=1/(t_{n}^{2}s^{2}(t_{n}))$, $\left(  \ref{E1}\right)
$ is equivalent to show that
\begin{equation}
t_{n}^{2}s^{2}(t_{n})\longrightarrow\infty. \label{E2}%
\end{equation}
By Theorem $\ref{order of s}$, $m(t_{n})\sim\psi(t_{n})$ and $s^{2}(t_{n}%
)\sim\psi^{\prime}(t_{n});$ this entails that $t_{n}\sim h(a_{n})$.

If $h\in R_{\beta}$, notice that
\[
\psi^{\prime}(t_{n})=\frac{1}{h^{\prime}(\psi(t_{n}))}=\frac{\psi(t_{n}%
)}{h\big(\psi(t_{n})\big)\big(\beta+\epsilon(\psi(t_{n}))\big)}\sim\frac
{a_{n}}{h(a_{n})\big(\beta+\epsilon(\psi(t_{n}))\big)}.
\]
holds. Hence we have
\[
t_{n}^{2}s^{2}(t_{n})\sim h(a_{n})^{2}\frac{a_{n}}{h(a_{n})\big(\beta
+\epsilon(\psi(t_{n}))\big)}=\frac{a_{n}h(a_{n})}{\beta+\epsilon(\psi(t_{n}%
))}\longrightarrow\infty.
\]

If $h\in R_{\infty}$, then $\psi(t_{n})\in\widetilde{R_{0}}$, it follows that
\[
t_{n}^{2}s^{2}(t_{n})\sim t_{n}^{2}\frac{\psi(t_{n})\epsilon(t_{n})}{t_{n}%
}=t_{n}\psi(t_{n})\epsilon(t_{n})\longrightarrow\infty,
\]
where the last step holds from condition $(\ref{3section1030})$. We have shown
$(\ref{E1}).$ Therefore
\[
B_{0}(\lambda_{n})=\frac{1}{\sqrt{2\pi}}\Big(1+o\left(  \frac{1}{n}\right)
\Big).
\]

By $(\ref{E2})$, $\lambda_{n}$ goes to $\infty$ as $a_{n}\rightarrow\infty$,
which implies further that $B_{3}(\lambda_{n})\rightarrow0$. On the other
hand, by Theorem \ref{order of s} , $\mu_{3}/s^{3}\rightarrow0$. Hence we
obtain from $(\ref{E3})$
\[
P(S_{1}^{n}\geq na_{n})=\frac{\Phi(t_{n})^{n}\exp(-nt_{n}a_{n})}{\sqrt{2\pi
n}t_{n}s(t_{n})}\Big(1+o\left(  \frac{1}{\sqrt{n}}\right)  \Big),
\]
which together with $(\ref{3virg01000})$ proves the claim.

\end{document}